\documentclass[11pt]{amsart}

\usepackage{fullpage}
\newtheorem{theorem}{Theorem}[section]
\usepackage{xcolor}
\newtheorem{lemma}[theorem]{Lemma}
\newtheorem{proposition}[theorem]{Proposition}
\newtheorem{corollary}[theorem]{Corollary}
\usepackage{amsmath}
\usepackage{graphicx} 
\usepackage{amsfonts}
\usepackage{amssymb}
\newtheorem{definition}[theorem]{Definition}
\newtheorem{example}[theorem]{Example}

\newtheorem{remark}[theorem]{Remark}
\numberwithin{equation}{section}
\usepackage{hyperref,xcolor}
\hypersetup{colorlinks=true, citecolor=red, linkcolor=red}
\usepackage{etoolbox}
\usepackage[font=small,labelfont=bf]{caption}
\usepackage{titletoc}
\begin{document}
\author[Beslikas]{Athanasios Beslikas}
\address{Doctoral School of Exact and Natural Studies,
Institute of Mathematics,
Faculty of Mathematics and Computer Science,
Jagiellonian University,
\L{}ojasiewicza 6, PL30348, Cracow, Poland}
\email{athanasios.beslikas@doctoral.uj.edu.pl}

\author[Sola]{Alan Sola}
\address{Department of Mathematics, Stockholm University, 106 91 Stockholm, Sweden}
\email{sola@math.su.se}

\title[Rational inner functions and Dirichlet spaces]{On the membership of two-variable rational inner functions in spaces of Dirichlet type}

\subjclass[2020]{Primary 32A08, 32A37, 46E22}
\keywords{Rational inner functions, Dirichlet spaces, Puiseux factorization}

\begin{abstract} 
We study membership of rational inner functions on the bidisk $\mathbb{D}^2$ in a scale of Dirichlet spaces considered by Bera, Chavan, and Ghara, and in higher-order variants of these spaces. We give a characterization for membership in terms of the geometric concept of contact order of a rational inner function at its singular points, and we further record some consequences and variants of our main result.
\end{abstract}

\maketitle
\section{Introduction} \label{Section 1}
\subsection{Preliminaries on Rational Inner Functions} 
Let $n \geq 1$ and let
\[\mathbb{D}^n=\{z=(z_1,\ldots, z_n)\in \mathbb{C}^n\colon |z_j|<1, \, j=1,\ldots, n\}\]
denote the unit polydisk. A bounded holomorphic function $\varphi\colon \mathbb{D}^n\to \mathbb{C}$ is said to be \textit{inner} if its non-tangential limits $\varphi^{*}(\zeta)$ exist at almost every point $\zeta$ belonging to the distinguished boundary
\[\mathbb{T}^n=\{\zeta=(\zeta_1,\ldots, \zeta_n)\colon |\zeta_j|=1, \, j=1,\ldots, n\},\]
and satisfy $|\varphi^{*}(\zeta)|=1$ whenever the limit exists.

Inner functions play a central role in the function theory of the polydisc, as well as in the study of operators acting on spaces of holomorphic functions in $\mathbb{D}^n$, see for instance \cite{Rudin,AMY}. The simplest inner functions are the rational ones; in other words, inner functions $\varphi=q/p$ where $q,p \in \mathbb{C}[z_1,\ldots, z_n]$. To avoid trivialities, we assume that $p$ and $q$ have no common factor, and that $p$ is {\it stable} with respect to $\mathbb{D}^n$ meaning that its zero set $\mathcal{Z}_p\cap \mathbb{D}^n=\varnothing$. In fact, $\varphi\colon \mathbb{D}^n\to \mathbb{C}$ is a rational inner function if and only if $\varphi$ admits the representation
\[\varphi(z)=e^{ia}z_1^{N_1}\cdots z_n^{N_n}\frac{\tilde{p}(z)}{p(z)},\]
where $a\in \mathbb{R}$, $N=(N_1,\ldots, N_n)\in \mathbb{N}$, and 
\[\widetilde{p}(z_1,z_2):=z_1^{m_1}\cdots z_n^{m_n}\overline{p\left(\frac{1}{\overline{z_1}},\ldots,\frac{1}{\overline{z_n}}\right)}\]
is the reflection of a stable polynomial $p$. See \cite[Chapter 5]{Rudin}. 

When $n=1$, the class of rational inner functions (abbreviated RIFs) coincides with finite Blaschke products in the unit disc $\mathbb{D}$, as can be seen by applying the fundamental theorem of algebra to the denominator polynomial $p$ and writing it as a product of factors of the form $1-\bar{\lambda}z$ where $\lambda \in \mathbb{D}$. In particular, a one-variable rational inner functions can be written as
\[\phi(z_1)=e^{ia}z_1^N\prod_{k=1}^{n}\frac{|\lambda_k|}{\lambda_k}\frac{z_1-\lambda_k}{1-\bar{\lambda}_kz}.\]
Finite Blaschke products have been studied extensively by many authors, and they are in many ways the nicest inner functions in one variable. For instance, every finite Blaschke product continues holomorphically to a strictly larger disk which contains $\mathbb{D}$ and in particular is a smooth function on $\overline{\mathbb{D}}$. A further example of the smoothness of finite Blaschke products is provided by a well-known result of Beurling which states that if an inner function $\phi\colon \mathbb{D}\to \mathbb{C}$ has  finite Dirichlet integral, that is
\[\int_{\mathbb{D}}|\varphi'(x+iy)|^2dxdy<\infty,\] 
then $\varphi$ is a finite Blaschke product (viz.  \cite[Theorem 7.6.9]{Ransford}). We refer the reader to \cite{GMR} for a comprehensive account of finite Blaschke products.

By contrast, an important feature of RIFs that distinguishes them from their one-dimensional counterparts is that they can have singularities on the distinguished boundary $\mathbb{T}^n:$
these occur at points $\zeta \in \mathbb{T}^n$ where $p$ and hence also $\tilde{p}$ vanish. Despite not being continuous on $\overline{\mathbb{D}^n}$ when singularities are present \cite{Pas17}, RIFs do have some degree of regularity: by a remarkable theorem of Knese \cite[Theorem C]{Knese2}, they do have a non-tangential limit at {\it every} point on the $n$-torus $\mathbb{T}^n$, and the value of this limit is unimodular. Thus, RIFs in two or more variables can be viewed as striking a balance between singular behavior and some residual smoothness. They exhibit a number of intimate connections with different branches of mathematics: 
\begin{enumerate}
\item  In operator theory in Hilbert function spaces (see for instance \cite{AMY}, \cite{Knese1} and \cite{Hilbert1}), they serve as generators of certain classes of shift-invariant subspaces, give rise to tractable types of pluriharmonic measures \cite{ClarkMeasures1,CLARK 2}, and induce composition operators \cite{Me}. 
\item They feature as solutions to interpolation problems on the polydisc (see for example the celebrated Agler's Pick Interpolation Theorem \cite{AMY}, or the seminal work of Kosi\'nski  \cite{Lukasz} on the three point Nevanlinna-Pick interpolation problem on the polydisc).
\item They have connections with concepts from classical algebraic geometry, specifically, with Puiseux factorization theorems and the study of the zero sets of stable polynomials in two or more complex dimensions, among other related results, see \cite{Pisa, Polonici, AJM}
\item Lastly, RIFs feature in function theory in the guise of Hilbert and Banach spaces of holomorphic functions, and their boundary behavior has been studied via membership of RIFs in the Dirichlet spaces on the bidisc as well as the that of their derivatives in Hardy spaces \cite{PLMS, AJM,Bergqvist}. 

\end{enumerate}
Our work is motivated mainly by (4), but the connections with the other directions above will be evident in the sequel. 

From now on, let us take $n=2$; the two-variable theory is generally better understood. There are several reasons for this, some of which will be reviewed shortly. Two important examples of RIFs on the bidisk with boundary singularities are furnished by
$$\kappa(z_1,z_2)=\frac{2z_1z_2-z_1-z_2}{2-z_1-z_2}$$
and
$$\psi(z_1,z_2)=\frac{4z_1^2z_2-z_1^2-3z_1z_2-z_1+z_2}{4-3z_1-z_2-z_1z_2+z_1^2},$$
where $(z_1,z_2)\in\mathbb{D}^2.$

The first of these functions has appeared frequently in the literature, in particular in work of Knese, whereas the second function (sometimes called the AMY example) appears in a paper by Agler, McCarthy, and Young \cite{AMY}. The reader is invited to verify the discontinuous nature of these functions at the point $(1,1)$ by examining them on the sets $\{(\zeta_1,1)\colon \zeta_1 \in \mathbb{T}\}$ and $\{(\zeta_1,\bar{\zeta}_1)\colon \zeta_1\in \mathbb{T}\}$. 

The papers \cite{PLMS,Pisa} study singularities of two-variable RIFs using several different approaches. One way of quantifying the singular nature of a RIF in $\mathbb{D}^2$ is via derivative integrability, for instance by determining the values of $\mathfrak{p}\geq 1$ such that $\frac{\partial \phi}{\partial z_j} \in L^{\mathfrak{p}}(\mathbb{T}^2)$. A main result of \cite{PLMS} is the complete characterization of derivative integrability in terms of a geometric quantity called contact order; see Section \ref{Section 3} for a formal definition. The paper \cite{Pisa} continues these investigations and proves, among other things, that the unimodular level sets of a RIF $\phi=\frac{\tilde{p}}{p}$ satisfy
\[\mathcal{C}_{\alpha}=\mathrm{clos}\{\zeta \in \mathbb{T}^2\colon \phi^{*}(\zeta)=\alpha\}=\{\zeta \in \mathbb{T}^2\colon \tilde{p}(\zeta)-\alpha p(\zeta)=0\}\] 
and can be parametrized by analytic functions. Refined local descriptions of the zero sets of the polynomials $p$ and $\tilde{p}$ as well as $\tilde{p}-\alpha p$ have been developed in \cite{Polonici} and \cite{TAMSKNESE}, with a view towards applications including analysis of RIF derivative singularities.

In higher dimensions, matters become more complicated as is explained in the papers \cite{AJM,Bergqvist}. For instance, different partial derivatives may exhibit different derivative integrability, level sets $\mathcal{C}_{\alpha}$ cannot in general be described as graphs of continuous functions, and cyclicity of a polynomial is determined not only by the size of its zero set but also its location within $\mathbb{T}^n.$ In light of these facts, we restrict our attention to two-variable RIFs in the present work.

\subsection{Dirichlet Spaces on the bidisc} 
In this paper, we contribute to the study of singular behavior of RIFs by examining these functions in norms of Dirichlet type, which involve partial derivatives of a RIF. Before we present our results, we need to review several definitions of Dirichlet spaces.

Let $\vec{\alpha}\in\mathbb{R}^2.$ The Dirichlet spaces $\mathfrak{D}_{\vec{\alpha}}$ on the bidisc can be defined 
using the Taylor series representation $f=\sum_{k=0}^{\infty}\sum_{\ell=0}^{\infty}a_{k\ell}z_1^kz_2^{\ell}$ of a holomorphic function  as
$$\mathfrak{D}_{\vec{\alpha}}(\mathbb{D}^2):=\Biggl\{ f\in \mathcal{O}(\mathbb{D}^2):\sum_{k=0}^{\infty}\sum_{\ell=0}^{\infty}(k+1)^{\alpha_1}(\ell+1)^{\alpha_2}|a_{k\ell}|^2<\infty\Biggr\}.$$
In the restricted range $\alpha_j\in(0,2)$ we have an equivalent definition of these spaces in terms of integrals, namely
\[\mathfrak{D}_{\vec{\alpha}}(\mathbb{D}^2):=\Biggl\{ f\in \mathcal{O}(\mathbb{D}^2):\int_{\mathbb{D}^2}\left|\frac{\partial^2(z_1z_2f(z_1,z_2))}{\partial z_1 \partial z_2}\right|^2dA_{\alpha_1}(z_1)dA_{\alpha_2}(z_2)<\infty\Biggr\},\] 
for a pair $\vec{\alpha}=(\alpha_1,\alpha_2)\in(0,2)^2,$ where
$dA_{\alpha_j}(z_j)=(1-|z_j|^2)^{1-\alpha_j}dA(z_j)$ and $dA(z_j)=dx_jdy_j/\pi$ is normalized Lebesgue measure on the unit disc.

The spaces $\mathfrak{D}_{\vec{\alpha}}$ were introduced by Kaptano\u{g}lu in \cite{Kaptanoglu} and have subsequently been studied by several authors; see for instance \cite{JP,Perfekt,Arcoetal} and the references provided there. These spaces are natural extension of the standard scale of one-variable weighted Dirichlet spaces $D_{\alpha}$ consisting of holomorphic functions $f=\sum_{k=0}^{\infty}a_kz^k$ in $\mathbb{D}$ whose Taylor coefficients satisfy
\[\|f\|^2_{\alpha}=\sum_{k=0}^{\infty}(k+1)^{\alpha}|a_k|^2<\infty.\] 
The spaces $\mathfrak{D}_{\vec{\alpha}}$ can be viewed as tensor products of the one-variable spaces, $\mathfrak{D}_{\vec{\alpha}}=D_{\alpha_1}\otimes D_{\alpha_2}$ and $\mathfrak{D}=\mathfrak{D}_{(1,1)}$ is characterized in terms of its invariance with respect to automorphisms of $\mathbb{D}^2$. In particular, the papers \cite{Cyclicity1, Cyclicity2} provide a complete characterization of which polynomials are cyclic in $\mathfrak{D}_{\alpha}$. Moreover, in the paper \cite{PLMS} the authors studied the membership of RIFs on the Dirichlet-type spaces of the bidisc, and the membership of the partial derivatives of RIFs on the Hardy space of the bidisc. In the latter case, a complete characterization is obtained in \cite{PLMS} but the membership problem in $\mathfrak{D}_{\alpha}$ remains open for general two-variable RIFs. The following conjecture is implicit in \cite{PLMS}, and is supported by \cite[Theorem 10.2]{PLMS} and other results in that paper and in \cite{Bergqvist}. 

\begin{itemize}
\item \textbf{Conjecture 1:} Let $\varphi=\frac{\tilde{p}}{p}$ be a Rational Inner Function on the bidisc with at least one singularity. Then $\varphi\notin\mathfrak{D}(\mathbb{D}^2).$\\
\end{itemize}
A refined version of \textbf{Conjecture 1} would ask for a characterization of $\mathfrak{D}_{\vec{\alpha}}$-membership of a RIF, again in terms of contact order.

The main goal of this paper is to characterize RIF membership in a different but related scale of Dirichlet-type spaces, namely those studied by Bera, Chavan, and Ghara \cite{Chavan}. For $f=\sum_{k=0}^{\infty}\sum^{\infty}_{\ell=0} a_{kl}z_1^kz_2^{\ell}$ a holomorphic function on the bidisc, they consider the following spaces using the coefficient norm
\begin{align}
\mathcal{D}(\mathbb{D}^2):=\Biggl\{f\in \mathcal{O}(\mathbb{D}^2): \sum_{k=0}^{\infty}\sum_{\ell=0}^{\infty}(k+\ell+1)|a_{k\ell}|^2<+\infty\Biggr\}. \notag 
\end{align}
The space $\mathcal{D}(\mathbb{D}^2)$ is a Hilbert space, and possesses an equivalent integral norm
\begin{align}
\|f\|^2_{\mathcal{D}(\mathbb{D}^2)}=\|f\|_{H^2(\mathbb{D}^2)}^2+\sup_{r<1}\int_{\mathbb{T}}\int_{\mathbb{D}}\left|\frac{\partial f}{\partial z_1}(z_1,re^{i\theta})\right|^2dA(z_1)d\theta  \\+\sup_{r<1}\int_{\mathbb{T}}\int_{\mathbb{D}}\left|\frac{\partial f}{\partial z_2}(re^{i\theta},z_2)\right|^2dA(z_2)d\theta \notag
\end{align}

In our paper, we consider weighted versions, and study the membership of RIFs in these spaces. Fix $\alpha\in(0,2).$ The \textit{mixed-norm Dirichlet-type spaces} $\mathcal{D}_{1,1,\alpha}(\mathbb{D}^2)$ consist of holomorphic functions $f\colon \mathbb{D}^2\to \mathbb{C}$ satisfying the norm boundedness condition
\begin{align}
\|f\|^2_{\mathcal{D}_{1,1,\alpha}(\mathbb{D}^2)}=\|f\|_{H^2(\mathbb{D}^2)}^2+\sup_{r<1}\int_{\mathbb{T}}\int_{\mathbb{D}}\left|\frac{\partial f}{\partial z_1}(z_1,re^{i\theta})\right|^2(1-|z_1|^2)^{1-\alpha}dA(z_1)d\theta\\+\sup_{r<1}\int_{\mathbb{T}}\int_{\mathbb{D}}\left|\frac{\partial f}{\partial z_2}(re^{i\theta},z_2)\right|^2(1-|z_2|^2)^{1-\alpha}dA(z_2)d\theta \notag .
\label{weightedBCGspaces}
\end{align}
As is explained in \cite{Chavan}, the space $\mathcal{D}$ is closely connected with cyclic analytic $2$-isometries. Indeed, $\mathcal{D}$ together with its weighted counterparts feature in a variant of Richter's representation theorem (see \cite{Ransford}) for pairs of operators \cite[Theorem 2.4]{Chavan}. Weighted spaces corresponding to \eqref{weightedBCGspaces} were recently studied by Nailwal and Zalar and by Torkinejad-Ziarati in \cite{NZprep, Pouriya1}, with a focus on polynomials cyclic with respect to shift operators.

The advantage of mixed-norm Dirichlet-type spaces in the present context is that in the integral representation of their norm, one integral is taken over the circle $\mathbb{T}.$ If the integrand is a RIF, say $\varphi(z_1,z_2),$ then for each $\zeta_2\in \mathbb{T}$ the slice function $z_1\mapsto \varphi_{\zeta_2}(z_1)=\varphi(z_1,\zeta_2)$ is a finite Blaschke product, whose zeros depend on the $\zeta_2$-variable. These mixed-norm spaces in general differ from the mixed-partial $\mathfrak{D}_\alpha$ spaces but we will see that there do exist connections between $\mathcal{D}_\alpha$ and $\mathfrak{D}_\alpha.$ 

Naturally, further generalizations of $\mathcal{D}_{1,1,\alpha}$ are possible, for instance by introducing weight functions, or by considering derivatives beyond first partials. In the latter direction, we consider higher-order weighted Dirichlet spaces on the bidisc, denoted by $\mathcal{D}_{n,m}(\mathbb{D}^2),$ which are endowed with the norm
\begin{align}
\|f\|^2_{\mathcal{D}_{m,n}(\mathbb{D}^2)}=\|f\|_{H^2(\mathbb{D}^2)}^2+&\sup_{r<1}\int_{\mathbb{T}}\int_{\mathbb{D}}\left|\frac{\partial^m f}{\partial z_1^m}(z_1,re^{i\theta})\right|^2dA(z_1)d\theta&&\\+&\sup_{r<1}\int_{\mathbb{T}}\int_{\mathbb{D}}\left|\frac{\partial^n f}{\partial z_2^n}(re^{i\theta},z_2)\right|^2dA(z_2)d\theta, \notag
\end{align}
This norm is equivalent to the coefficient norm
$$\|f\|^2_{\mathcal{D}_{m,n}(\mathbb{D}^2)}\approx \sum_{k=0}^{\infty}\sum_{\ell=0}^{\infty}(k^{2m-1}+\ell^{2n-1}+1)|a_{k\ell}|^2.$$
For $n=m=1$ we recover the space $\mathcal{D}(\mathbb{D}^2)=\mathcal{D}_{1,1}(\mathbb{D}^2)$ as defined above.

In the present paper we essentially resolve the problem of membership of RIFs in the spaces $\mathcal{D}_{1,1,\alpha}(\mathbb{D}^2)$. As we will later observe, equivalently, we characterize the membership of RIFs in the intersection $\mathfrak{D}_{(\alpha,0)}(\mathbb{D}^2)\cap \mathfrak{D}_{(0,\alpha)}(\mathbb{D}^2)$ of anisotropic Dirichlet spaces. Moreover, we address the membership problem for a restricted class of RIFs in the higher order derivative spaces $\mathcal{D}_{n,m}(\mathbb{D}^2)$, and thus obtain some further albeit somewhat circumstantial support for Conjecture 1.

\section{Main results} \label{section 2}
In \cite{PLMS}, membership of the partial derivatives of RIFs in the $H^{\mathfrak{p}}(\mathbb{D}^2)$ was completely characterized in terms of a geometric quantity called \textit{contact order}, which informally measures the speed at which the zero sets of the slice functions $z_1\mapsto \varphi(z_1,\zeta_2)$ and $z_2\mapsto \varphi(\zeta_1,z_2)$ approach $\mathbb{T}$. We review this notion in detail in Section \ref{Section 3} but emphasize here that this is a characteristic of the underlying polynomial $p$.

The main result we obtain in the present paper is a characterization for the membership in $\mathcal{D}_{1,1,\alpha}(\mathbb{D}^2)$ of RIFs. The statement is again formulated in terms of the contact order of a RIF, and is strikingly similar to the corresponding result concerning $H^{\mathfrak{p}}$-membership of partial derivatives, see \cite[Theorem 4.1]{PLMS}. 

\begin{theorem} Let $p\in\mathbb{C}[z_1,z_2]$ be stable on $\mathbb{D}^2.$  Then, $\varphi=\frac{\tilde{p}}{p}\in \mathcal{D}_{1,1,\alpha}(\mathbb{D}^2)$ if and only if $\alpha<1+\frac{1}{K},$ where $K$ is the contact order of $\varphi$.
\label{membership}
\end{theorem}
Let us point out that if $\varphi$ has no singularities, then $K=0$ and the statement of the theorem is trivial. 

The techniques used to obtain Theorem \ref{membership} are reminiscent of those in \cite{PLMS}, with some important differences in the part of the proof that establishes necessity.

One immediately deduces the following two corollaries.
\begin{corollary} Let $\varphi$ be as in Theorem \ref{membership}. Then, $\varphi\in \mathfrak{D}_{(\alpha,0)}(\mathbb{D}^2)\cap \mathfrak{D}_{(0,\alpha)}(\mathbb{D}^2)$ if and only if $\alpha<1+\frac{1}{K}.$
\label{anisotropicchar}
\end{corollary}
The proof is an application of Proposition \ref{claDir}.
\begin{corollary} Let $\varphi$ be a RIF. If $\alpha<\frac{1}{2}+\frac{1}{2K},$ then $\varphi\in \mathfrak{D}_{\alpha}(\mathbb{D}^2).$
\label{sufficient0}
\end{corollary}
Let us briefly comment on these corollaries. The suffiency part of Corollary \ref{anisotropicchar} is known from \cite{PLMS}, but we underline that the necessity part was not proved in that paper. Moreover, Corollary \ref{sufficient0} follows as a corollary from the characterization of $H^{\mathfrak{p}}$-containment of partial derivatives. The interest here lies in the fact that our proof directly uses Dirichlet-type norms. 

Another simple Corollary we obtain is
\begin{corollary} Let $\varphi$ be a RIF. Then $\varphi \in \mathcal{D}_{1,1,1}(\mathbb{D}^2).$
\label{easy}
\end{corollary}

In other words, all RIFs belong to the \textit{mixed norm Dirichlet space} $\mathcal{D}_{1,1,1}(\mathbb{D}^2),$ whether they have singularities or not, and therefore, all RIFs belong to $\mathfrak{D}_{1/2}(\mathbb{D}^2).$ 

The main difficulty one faces when investigating membership of a RIF in $\mathfrak{D}_{\alpha}$ is the presence of mixed partials. This leads to the question of what happens with the membership of RIFs on the higher-order Dirichlet-type spaces that we defined above. In this direction, we obtain the following result, which showcases an interesting phenomenon:
\begin{theorem} Let $p\in\mathbb{C}[z_1,z_2]$ be of bidegree $(1,1),$ stable on $\mathbb{D}^2$ and vanishing only on $\tau\in\mathbb{T}^2.$ Let also $\varphi=\frac{\widetilde{p}}{p}.$ Then, $\varphi\in \mathcal{D}_{n,m}(\mathbb{D}^2)$ if and only if $\max(m,n)<2.$
Moreover, the condition $n<2$ is necessary whenever $p$ is of bidegree $(1,k)$ and $m<2$ is necessary whenever $p$ is of bidegree $(k,1)$, for integers $k\ge1.$
\label{derivativedirichlet}
\end{theorem}

\subsection{Structure of the paper } Section \ref{Section 3} contains a brief review of concepts and results we use in our proofs, which are in turn given in Section \ref{section 4}. Section \ref{section 5}  collects some miscellaneous material, including two-variable weighted Douglas formulas and results that relate Dirichlet-space membership to properties of Agler decompositions. Finally, in Section \ref{section 6} we present several examples that illustrate our results.
\subsection{Notation} For the entirety of the paper, we shall use the notation $A\approx B$ for inequalities of the form $C_1B\leq A\leq C_2B $ for some positive constants $C_1,C_2>0.$ Moreover, we denote by $A\lesssim B$ inequalities of the form $A\leq CB,$ for some positive constant $C>0.$
\section{Contact order, factorization theorems, and past results} \label{Section 3}
Before we discuss the proofs of the main results, we review some background and tools used in the study of RIFs. Let us first discuss the concept of \textit{contact order}. We refer the reader to \cite{PLMS,Polonici} for in-depth treatments. Consider a rational inner function $\varphi(z_1,z_2)=\widetilde{p}(z_1,z_2)/p(z_1,z_2),$ and assume that this function has a singularity $(\tau_1,\tau_2)\in\mathbb{T}^2.$ In other words, we assume that $p(\tau_1,\tau_2)=0$ and then necessarily also $\tilde{p}(\tau_1,\tau_2)=0$. 

For each $\zeta_2\in\mathbb{T},$ the function \[\mathbb{D}\ni z_1\mapsto \varphi_{\zeta_2}(z_1)\] is a finite Blaschke product, with discrete zero set $\mathcal{Z}= \{a_1(\zeta_2),...,a_N(\zeta_2)\}\subset \mathbb{D}.$ Note that $N\leq \deg_{z_1}p $.
The same holds if we fix the $z_1$-variable on the circle $\mathbb{T}.$ As is explained in \cite{PLMS}, the distance of the zero set of $\phi_{\zeta_2}$ to the unit circle $\mathbb{T}$ is controlled by $\min_{1\leq j\leq N}\{1-|a_j(\zeta_2)|\}.$ As $\zeta_2 \to \tau_2 \in \mathbb{T}$, some of the zeros of the finite Blaschke product $\phi_{\zeta_2}$ tend to $\tau_1$. A natural question is at what ``speed" this phenomenon occurs. The answer to this question is given in the following theorem.
\begin{theorem}(Theorem 3.5. of \cite{PLMS}) 
\label{def:loccont}
Let
$\varphi=\frac{\tilde p}{p}$ 
be a rational inner function on $\mathbb{D}^2$ with $\deg \varphi=(m,n)$ and a
singularity on $\mathbb{T}^2$ with $z_2$-coordinate $\tau_2$. Then there exists a
rational number $K>0$ such that
$$\epsilon(\varphi,\zeta_2)=\min_{1\leq j\leq n}\{1-|a_j(\zeta_2)|\}\approx |\tau_2-\zeta_2|^K,$$
for all $\zeta_2\in\mathbb{T}$ in a neighborhood $U$ of $\tau_2$. The number $K$ is
called the $(z_1,\tau_2)$-contact order of $\varphi$.
\end{theorem}

This result allows one to formally define the contact order of a RIF, even in the case of multiple singularities. More concretely,

\begin{definition}(Definition 3.6 of \cite{PLMS}) Let $\varphi=\frac{\tilde p}{p}$ be a rational inner function on $\mathbb{D}^2$ with $\deg \varphi=(m,n)$ and let
$\tau_{2,1},\ldots,\tau_{2,J}$ denote the distinct $z_2$-coordinates of the singularities of
$\varphi$. Then the $z_1$-contact order of $\varphi$ is the number $K^1$ defined by

$$K^1:=\max\bigl\{(z_1,\tau_{2,j})\text{-contact order of }\varphi : 1\le j\le J\bigr\}.$$
\end{definition}

The definition of the $z_2-$contact order $K^2$ of $\varphi$ is set up in an analogous manner.
The papers \cite{Pisa,Polonici}, among others, present more systematic studies of contact order. One basic 
result is that the $z_1$- and $z_2$-contact orders are
equal, meaning that it makes sense to speak of the {\it contact order} $K=K^1=K^2$ of $\varphi$.
Another important property of the contact order $K$ is that it is an even integer, $K=2L$ for some $L\in \mathbb{N}$. Very roughly speaking, these result are consequences of Puiseux factorization theorems for stable polynomials. A version of such a Puiseux factorization result is presented in \cite{Polonici} in the setting of the bi-upper halfplane $\mathbb{H}^2=\mathbb{H}\times \mathbb{H}=\{(z_1,z_2)\in \mathbb{C}^2\colon \Im z_j>0, \, j=1,2\}.$ Note that any stable polynomial in $\mathbb{D}^2$ can be converted to a stable polynomial in $\mathbb{H}^2$ and vice versa via coordinatewise Cayley transforms 
\[\beta \colon \mathbb{D}\to \mathbb{H}, \quad \beta(z_j)=\frac{1+iz_j}{1-iz_j}, \quad j=1,2,\]
and their inverses.
Specifically, if $p\in \mathbb{C}[z_1,z_2]$ is stable in $\mathbb{D}^2$ and of bidegree $(m,n)$, then the polynomial
\[(1-iz_1)^m(1-iz_2)^np\left(\frac{1+iz_1}{1-iz_1},\frac{1+iz_2}{1-iz_2} \right)\]
is stable in $\mathbb{H}$. See \cite{Pisa,Polonici} for details, including the fact that contact order is invariant under Cayley transforms.

Let us state here the specific result which is Theorem 1.2. of the mentioned paper.
\begin{theorem}(Theorem 1.2. of \cite{Polonici}) Let $p(z_1,z_2)\in\mathbb{C}[z_1,z_2]$ have no zeros on $\mathbb{H}^2$ and no common factors with $\overline{p}.$ Suppose $p$ vanishes to order $M$ at $(0,0).$ Then, there exists natural numbers $L_1,...,L_M\ge 1$ and real coefficient polynomials $q_1(x),...,q_M(x)\in\mathbb{R}[x]$ satisfying 
\begin{itemize}
\item $q_j(0)=0$ 
\item $q_j'(0)>0$
\item $\deg q_j<2L_j$
\end{itemize}
for $j=1,...,M$ such that if we define
$$[p](z_1,z_2)=\prod_{j=1}^M(z_1+q_j(z_2)+iz_2^{2L_j}),$$
then $p/[p]$ and $[p]/p$ are bounded in a punctured neighborhood of $(0,0)$ in $\mathbb{R}^2.$
\label{KneseLocal}
\end{theorem} Note that $K_0=\max\{2L_1, \cdots, 2L_M\}$ gives the $(0,0)$-contact order of $\varphi$.
As before then, $K=\{K_{\tau}\colon \tau \in \mathcal{Z}_{p}\cap \mathbb{R}^2\}$ is the global contact order, or simply \textit{contact order}. Note that the polynomial $[p]$ as defined above is uniquely determined by $p$ (since it arises from a truncation of the full Puiseux expansions of $p$).

The result that connects the contact order as an algebro-geometric quantity and the membership of partial derivatives of RIFs on the Hardy spaces $H^\mathfrak{p}(\mathbb{D}^2)$ is the following.

\begin{theorem} (Theorem 4.1. of \cite{PLMS}) 
Let $\varphi=\frac{\widetilde{p}}{p}$ be a RIF. Then $\frac{\partial\varphi}{\partial z_j}, j=1,2,$ belong to $H^\mathfrak{p}(\mathbb{D}^2)$ if and only if $\mathfrak{p}<1+\frac{1}{K}.$
\label{HardyMembership}
\end{theorem}
The proof of the said necessity part is independent of the proof of Theorem \ref{HardyMembership}. At its core, the proof utilizes the hyperbolic pseudometric $$\rho_{\mathbb{D}}(z,w)=\left|\frac{z-w}{1-\overline{w}z}\right|$$ where $z,w\in\mathbb{D}$ (and its counterpart on the upper half-plane $\mathbb{H}=\{z\in \mathbb{C}:\mathrm{Im}z>0\}$ and the local boundedness Theorem \ref{KneseLocal} we stated above.
The factorization of $p$ yields factors of the form $y_j(z_2)=q(z_2)+z_2^{2L_j}\Psi_j(z_2^{1/M_j})$ for $\Psi$ an analytic function with $\mathrm{Im}(\Psi_j(0))>0.$ In our situation it suffices to use the local model provided by $[p]$.
\section{Proofs of the main results} \label{section 4} In the present section the proofs of the main results will be presented.  First, we prove Theorem \ref{membership}. Then, we prove Proposition \ref{claDir}. This Proposition will allow us to obtain Corollaries \ref{sufficient0} and \ref{anisotropicchar}. We finish this Section with the proof of Theorem \ref{derivativedirichlet}.

Let us begin with the proof of Theorem \ref{membership}.
\begin{proof}[Proof of Theorem \ref{membership}] First, we prove the sufficiency. Before we proceed, the following Lemma is required. We state it for $n-$th order derivatives for convenience, as we will apply it to both of our proofs.
\begin{lemma} For $z_1\in\mathbb{D}, \varphi=\frac{\tilde{p}}{p}$ a RIF, and $n\in\mathbb{N}$ the following holds
$$\sup_{r<1}\frac{1}{2\pi}\int_{-\pi}^{\pi}\left|\frac{\partial^n\varphi}{\partial z_1^n}(z_1,re^{i\theta})\right|^2d\theta=\frac{1}{2\pi}\int_{-\pi}^{\pi}\left|\frac{\partial^n\varphi}{\partial z_1^n}(z_1,e^{i\theta})\right|^2d\theta.$$
\end{lemma}
\begin{proof} For $f\in H^2(\mathbb{D}),$
$$\sup_{r<1}\frac{1}{2\pi}\int_{-\pi}^{\pi}\left|f(re^{i\theta})\right|^2d\theta=\frac{1}{2\pi}\int_{-\pi}^{\pi}\left|f(e^{i\theta})\right|^2d\theta,$$
with the right-hand side interpreted in terms of non-tangential limits.
We observe that $$\frac{\partial^n\varphi}{\partial z_1^n}(z_1,e^{i\theta})=\frac{Q_n(z_1,e^{i\theta})}{p^{n+1}(z_1,e^{i\theta})},$$ and note that the fraction is an analytic function of $\theta$ because the denominator is non-vanishing on $\mathbb{D},$ as a consequence of $p$ being stable. Setting $f(re^{i\theta})=\frac{\partial^n \varphi}{\partial z_1^n}(z_1,re^{i\theta})$ finishes the proof.
\end{proof}
We shall also recall the well-known Rudin-Forelli estimates (see Theorem 1.7 in \cite{HedenmalmBook}) on the disc. For $\beta\in(-1,+\infty)$ and $\gamma\in(-\infty,\infty)$, we have
$$\int_{\mathbb{D}}\frac{(1-|z|^2)^{\beta}}{|1-\overline{w}z|^{2+\beta+\gamma}}dA(z)\approx \begin{cases}1 &, \gamma<0\\
-\log(1-|w|^2) &,\gamma=0\\
(1-|w|^2)^{-\gamma} &,\gamma>0
\end{cases}$$
as $|w|\to 1.$
Now, we are ready to continue with the proof. Due to symmetry of the RIFs, by the previous Lemma and the definition of the $\mathcal{D}_{1,1,\alpha}(\mathbb{D}^2)-$norm, it suffices to estimate from above the integral
$$I(\varphi)=\int_{\mathbb{T}}\int_{\mathbb{D}}\left|\frac{\partial \varphi}{\partial z_1}(z_1,e^{i\theta})\right|^2(1-|z_1|^2)^{1-\alpha}dA(z_1)d\theta.$$
Fixing the $z_2$-variable on the circle, we get 
$\varphi(z_1,\zeta_2)=\varphi_{\zeta_2}(z_1)$ where $\varphi_{\zeta_2}(z_1)$ is a finite Blaschke product in the $z_1-$variable. Therefore, 
$$\frac{\partial \varphi}{\partial z_1}(z_1,\zeta_2)=\varphi'_{\zeta_2}(z_1), \quad z_1\in\mathbb{D}.$$
By the formula for the derivative of a finite Blaschke product (see \cite[Lemma 4.2]{PLMS} for details), we obtain $$\left|\frac{\partial \varphi}{\partial z_1}(z_1,\zeta_2)\right|^2\leq \left|\sum_{j}|b'_{a_j(\zeta_2)}(z_1)|\right|^2\lesssim\sum_{j}|b'_{a_j(\zeta_2)}(z_1)|^2,$$
where $b_{a_j}(z_1)=\frac{z_1-a_j(\zeta_2)}{1-\overline{a_j(\zeta_2)}z_1}$ are the individual Blaschke factors making up $\varphi_{\zeta_2}$.
Note here that the last inequality is justified by applying the simple inequality $(a+b)^2\leq 2(a^2+b^2).$
Hence,
\begin{align}
\int_{\mathbb{T}}\int_{\mathbb{D}}\left|\frac{\partial \varphi}{\partial z_1}(z_1,e^{i\theta})\right|^2dA_{\alpha}(z_1)d\theta\leq& \sum_{j}\int_{\mathbb{T}}\int_\mathbb{D}|b'_{a_j(e^{i\theta})}(z_1)|^2(1-|z_1|^2)^{1-\alpha}dA(z_1)d\theta  &&\\
=&\sum_{j}\int_{\mathbb{T}}\int_{\mathbb{D}}\frac{(1-|a_j(e^{i\theta})|^2)^2(1-|z_1|^2)^{1-\alpha}}{|1-\overline{a_j(e^{i\theta})}z_1|^4} dA(z_1)d\theta \notag \\
=&\sum_{j}\int_{\mathbb{T}}(1-|a_j(e^{i\theta})|^2)^2\left(\int_{\mathbb{D}}\frac{(1-|z_1|^2)^{1-\alpha}}{|1-\overline{a_{j}(e^{i\theta})}z_1|^4}dA(z_1)\right)d\theta \notag \\
\approx&\sum_{j}\int_{\mathbb{T}}(1-|a_{j}(e^{i\theta})|^2)^{2}\cdot \frac{1}{(1-|a_j(e^{i\theta})|^2)^{1+\alpha}}d\theta  \notag 
\end{align}
with the last asymptotic equality coming from the Rudin-Forelli estimates, after choosing $\beta=1-\alpha,$ $\gamma=1+\alpha$ and writing $4=2+\beta+\gamma=2+(1-\alpha)+(1+\alpha).$ Now assume that the singularity of the RIF that realizes its contact order is at the point $\tau=(\tau_1,\tau_2)\in \mathbb{T}^2.$
From the results in \cite{PLMS} (viz. in particular pp. 295-298), it is known that 
$$\int_{\mathbb{T}}\sum_{j}(1-|a_{j}(e^{i\theta})|^2)^{1-\alpha}d\theta\approx\int_{U}|\tau_2-\zeta_2|^{K(1-\alpha)}|d\zeta_2|,$$
where $U$ is a neighborhood of $\tau_2\in\mathbb{T}.$
The integral on the right hand side converges if and only if $K(1-\alpha)>-1,$ where $K$ is the contact order of the RIF. This in turn, implies convergence of the integral provided $\alpha<1+\frac{1}{K}.$ These observations finish the ``sufficiency" part of the proof. 

For the necessity part, we need a lower estimate on the integrals in the $\mathcal{D}_{1,1,\alpha}$ norm that allows us to compare it with 
$$\int_{U}|\tau_2-\zeta_2|^{K(1-\alpha)}|d\zeta_2|,$$
where $U$ again is a small neighborhood around the $\tau_2$-coordinate of the singularity. As before, we take a detour via one-variable functions, with the following lemma playing a crucial role.
\begin{lemma} Let $B(z)$ be a finite Blaschke product with $N$ distinct, simple zeros $a_1,...,a_N\in\mathbb{D}.$ Then the following estimate holds
$$\int_{\mathbb{D}}|B'(z)|^2(1-|z|^2)^{1-\alpha}dA(z)\ge  (1-|a_j|^2)^{1-\alpha}\prod_{j\neq k}\left|\frac{a_k-a_j}{1-\overline{a_k}a_j}\right|^2, \quad  j=1,\ldots, N.$$
In particular, 
\[\int_{\mathbb{D}}|B'(z)|^2(1-|z|^2)^{1-\alpha}dA(z)\ge
\max_{j=1,\ldots, N}\left\{(1-|a_j|^2)^{1-\alpha}\prod_{j\neq k}\left|\frac{a_k-a_j}{1-\overline{a_k}a_j}\right|^2\right\}.
\label{estimate}
\]
\end{lemma}
\begin{proof} The proof is essentially an application of bounded point evaluation property for weighted Bergman spaces, see \cite[Proposition 1.1]{HedenmalmBook}. Namely, $$(1-|a_j|^2)^{3-\alpha}|B'(a_j)|^2\leq \|B'(z)\|^2_{A^2_{1-\alpha}(\mathbb{D})}.$$
Substituting $$|B'(a_j)|=\frac{1}{1-|a_j|^2}\prod_{j\neq k}\left|\frac{a_j-a_k}{1-\overline{a_k}a_j}\right|,$$ we obtain the desired statements. 
\end{proof}
\begin{lemma} \label{generousReferee} Let $B$ be a finite Blaschke product and let $B_*$ be a finite Blaschke product with the same zeros as $B$, but with each zero having multiplicity $1$. Then, for each $\alpha\in(0,2),$ 
$$\int_{\mathbb{D}}|B'(z)|^2(1-|z|^2)^{1-\alpha}dA(z)\ge \int_{\mathbb{D}}|B_{*}'(z)|^2(1-|z|^2)^{1-\alpha}dA(z).$$
\end{lemma}
\begin{proof}
The initial finite Blaschke product can be factored as $B(z)=B_*(z)\cdot R(z),$ where $R$ is another (potentially constant) finite Blaschke product, built using the repeated Blaschke factors of $B.$ The backwards shift operator, defined as
$$S^*(f):=\frac{f(z)-f(0)}{z}, \quad z\in \mathbb{D},$$
is a contraction in $D_{\alpha}(\mathbb{D}),$ the Dirichlet-type spaces on the unit disc. To see this, observe that
$$\|S^*f(z)\|^2_{D_{\alpha}(\mathbb{D})}=\sum_{n=0}^{\infty}(n+1)^{\alpha}|a_{n+1}|^2=\sum_{m=1}^{\infty}m^{\alpha}|a_m|^2\leq \sum_{m=1}^{\infty}(m+1)^{\alpha}|a_m|^2=\|f\|^2_{D_{\alpha}(\mathbb{D})}.$$
Let us define the \textit{co-analytic Toeplitz operator} $$T_{\overline{R}}f=P_{+}(\overline{R}f)\colon D_{\alpha}(\mathbb{D})\to D_{\alpha}(\mathbb{D})$$
with symbol $\overline{R};$ here $P_{+}$ denotes orthogonal projection onto $D_{\alpha}(\mathbb{D})$.
Since $R$ is itself a finite Blaschke product, we can write $$R(z)=\sum_{n=0}^{\infty}b_{n}z^n, \quad z\in \mathbb{D},$$ and as a consequence
$$T_{\overline{R}}(f)=\sum_{n=0}^{\infty}\overline{b_n}(S^*)^n(f).$$ By von Neumann's inequality for functions in the disk algebra, we obtain
$$\|T_{\overline{R}}(f)\|^2_{D_{\alpha}(\mathbb{D})}\leq \|R\|_{H^{\infty}(\mathbb{D})}^2\|f\|_{D_{\alpha}(\mathbb{D})}^2.$$
Moreover, on the circle $\mathbb{T},$ we observe that $\overline{R}B=B_{*}$ and therefore
$T_{\overline{R}}(B)=B_{*}.$ Setting $f=B$ finishes the proof.
\end{proof} 
\begin{remark}
An earlier version of our paper imposed the additional assumption that $p$ be locally square free, meaning that $p$ have no repeated factors in its local Puiseux factorizations at zeros on $\mathbb{T}$. Lemma \ref{generousReferee}, which is likely well-known to specialists, was very generously suggested to us by an anonymous referee, and has allowed us to remove this additional assumption to get a full characterization.  
\end{remark}
Bypassing the case of repeated zeros, by combining Lemmas \ref{estimate}, \ref{generousReferee} we arrive at a non-trivial lower bound for the $\mathcal{D}_{1,1,\alpha}-$norm. We have the following estimate
\begin{align}
||\varphi||^2_{\mathcal{D}_{1,1,\alpha}(\mathbb{D}^2)}=&\int_{\mathbb{T}}\int_{\mathbb{D}}\left|\frac{\partial \varphi}{\partial z_1}(z_1,\zeta_2)\right|^2(1-|z_1|^2)^{1-\alpha}dA(z_1)|d\zeta_2|\notag&&\\
=&\int_{\mathbb{T}}\int_{\mathbb{D}}|\varphi'_{\zeta_2}(z_1)|^2(1-|z_1|^2)^{1-\alpha}dA(z_1)|d\zeta_2|&&\\
\gtrsim& \int_{\mathbb{T}}\max_{1\leq j\leq N}\Biggl\{(1-|a_j(\zeta_2)|^2)^{1-\alpha}\prod_{j\neq k}\left|\frac{a_{j}(\zeta_2)-a_{k}(\zeta_2)}{1-\overline{a_{k}(\zeta_2)}a_{j}(\zeta_2)}\right|^2\Biggr\}|d\zeta_2| \notag
\end{align}
As we can observe here, the product inside the integration is a product of hyperbolic pseudodistances of the branches of the zero set $\mathcal{Z}_{\widetilde{p}}.$ One needs to see how these  pseudodistances behave when the $z_2$-variable approaches $\tau_2\in\mathbb{T}.$ The idea is to prove that whenever we approach the $z_2$-coordinate of the singularity, the pairwise pseudodistances tend to a constant bounded away from zero. This, along with a simple continuity argument would lead to an estimate of the form
\begin{align}
\int_{\mathbb{T}}\max_{1\leq j\leq N}\Biggl\{(1-|a_j(\zeta_2)|^2)^{1-\alpha}\prod_{j\neq k}\left|\frac{a_{j}(\zeta_2)-a_{k}(\zeta_2)}{1-\overline{a_{k}(\zeta_2)}a_{j}(\zeta_2)}\right|^2\Biggr\}|d\zeta_2|\gtrsim \int_{U}|\tau_2-\zeta_2|^{K(1-\alpha)}|d\zeta_2|
\end{align}
where $U$ is a neighborhood around $\tau_2\in\mathbb{T}$ and $K$ is the global contact order. This, in turn, would provide us the desired result. In order to be able to prove this rigorously, we will transfer the problem to the upper halfplane $\mathbb{H}$ and we will exploit the mentioned local theorem (see \cite[Theorem 1.2]{Polonici}). Using the Cayley transform $\beta:\mathbb{H}\to \mathbb{D}$ as in Section \ref{Section 3} , we can map every branch $a_j(\zeta_2)$ to a corresponding branch $y_j(z_2).$ Thus we are reduced to understanding how the pseudodistances
$$\rho_{\mathbb{H}}(y_j(z_2),y_k(z_2))=\left|\frac{y_j(z_2)-y_k(z_2)}{y_j(z_2)-\overline{y_k(z_2)}}\right|$$
behave as $z_2$ approaches the origin, when $z_2$ is restricted to the real line (respectively $\zeta_2\in\mathbb{T}.$)

As we have mentioned in Section 3, in a punctured neighborhood of the point $(0,0),$ Theorem \ref{KneseLocal} assures that the functions $p/[p]$ and $[p]/p$ are bounded. Recall that 
\begin{align}\label{loc}
[p](z_1,z_2)=\prod_{j}^{N}(z_1+q_j(z_2)+iz_2^{2L_j}),
\end{align}

with the hypothesis that no two branches of the zero set have equal initial segments $q_j(z_2)$ and equal contact orders $2L_j.$  It suffices to work with the local model and ignore the analytic part $\Psi_j$ of every branch (see \cite{TAMSKNESE}), as we will work in a neighborhood of $0$. We need to consider three cases.

\textbf{Case 1: Identical initial segments-different contact orders.}
Take two distinct branches with $q_i=q_j$ and $2L_j<2L_k.$ The branches for us are
$y_j=q_j(z_2)+iz_2^{2L_j}$
and $y_k=q_j(z_2)+iz_2^{2L_k},$ where $z_2$ is restricted to $\mathbb{R}$ (as the $z_2-$variable is restricted on the boundary), after applying \ref{loc},
Then, we see that 
$$
\rho_{\mathbb{H}}(y_j,y_k)=\left|\frac{q_{j}(z_2)+iz_2^{2L_j}-q_{j}(z_2)-iz_2^{2L_k}}{q_{j}(z_2)+iz_2^{2L_j}-q_{k}(z_2)+iz_2^{2L_k}}\right|=\left|\frac{iz_2^{2L_j}(1-z_2^{2L_k-2L_j})}{iz_2^{2L_j}(1+z_2^{2L_k-2L_j})}\right|\to 1, \quad \textrm{as}\, \,z_2 \to 0.
$$

\textbf{Case 2: Different initial segments-same contact order} Here we need to clarify what exactly is meant by ``different" initial segments. Let us be more precise: $q_j,q_k\in \mathbb{R}[z_2]$ are two polynomials with real coefficients which differ in at least one coefficient. This means that 
\begin{align}
\rho_{\mathbb{H}}(y_j,y_k)&=\left|\frac{q_{j}(z_2)+iz_2^{2L_j}-q_{k}(z_2)-iz_2^{2L_j}}{q_{j}(z_2)+iz_2^{2L_j}-q_{k}(z_2)+iz_2^{2L_j}}\right|&&\\
&\approx \left|\frac{q_j(z_2)-q_k(z_2)}{q_j(z_2)-q_k(z_2)+2iz_2^{2L_j}}\right|=1 \notag, \quad \textrm{as}\,\, z_2\to 0,
\end{align}
since $1<\deg(q_j-q_k)<2L_j.$  

\textbf{Case 3: Different initial segments-different contact orders}\\
Assume that $2L_j<2L_k$ and $q_j$ differs from $q_k$ in at least one coefficient; recall that $q_j(0)=q_k(0)=0$ and $q_j'(0),q'_k(0)>0.$ Without loss of generality, assume that $q_j'(0)\neq q_k'(0)$ (otherwise we can consider the lowest degree terms of $q_k$ and $q_j$ that do differ, and proceed similarly). Locally, we have 
\begin{align}
\rho_{\mathbb{H}}(y_j,y_k)&=\left|\frac{q_{j}(z_2)+iz_2^{2L_j}-q_{k}(z_2)-iz_2^{2L_k}}{q_{j}(z_2)+iz_2^{2L_j}-q_{k}(z_2)+iz_2^{2L_k}}\right| &&\\
&=\left|\frac{z_2(Q(z_2)+iz_2^{2L_j-1}(1-z_2^{2L_k-2L_j}))}{z_2(Q(z_2)+iz_2^{2L_j-1}(1+z_2^{2L_k-2L_j}))}\right| \notag&&\\
&=\left|\frac{Q(0)+\text{higher order terms}+iz_2^{2L_j-1}(1-z_2^{2L_k-2L_j})}{Q(0)+\text{higher order terms}+iz_2^{2L_j-1}(1+z_2^{2L_k-2L_j})}\right|\notag&&\\
&\approx\left|\frac{Q(0)}{Q(0)}\right|= 1, \notag
\end{align}
with $Q(z_2)$ being the polynomial that emerges after factoring out $z_2$ from the difference $q_j(z_2)-q_k(z_2). $ This implies that $Q(0)\neq 0,$ as it is the first non-zero coefficient after factoring out $z_2.$ 

At this point, we shall estimate the integral 
$$I=\int_{\mathbb{T}}\max_{1\leq j\leq N}\Biggl\{(1-|a_j(\zeta_2)|^2)^{1-\alpha}\prod_{j\neq k}\left|\frac{a_{j}(\zeta_2)-a_{k}(\zeta_2)}{1-\overline{a_k(\zeta_2)}a_{j}(\zeta_2)}\right|^2\Biggr\}|d\zeta_2|.$$
Set $$P(\zeta_2)=\prod_{j\neq k}\left|\frac{a_{j}(\zeta_2)-a_{k}(\zeta_2)}{1-\overline{a_{k}(\zeta_2)}a_{j}(\zeta_2)}\right|,$$
and create 3 sets of indices, denoted by $S_1,S_2,S_3.$ Each set contains the indices that correspond to every case we considered above. Then, write
\begin{align}
P(\zeta_2)=&\prod_{k\in S_1}\rho_{\mathbb{D}}(a_j,a_k) \prod_{k\in S_2}\rho_{\mathbb{D}}(a_j,a_k) \prod_{k\in S_3}\rho_{\mathbb{D}}(a_j,a_k) &&\\
=&P_1(\zeta_2)P_2(\zeta_2)P_3(\zeta_2). \notag
\end{align}
It is clear by the case study above, that the product in question will be bounded away from zero, in a neighborhood $U\subset \mathbb{T}$ around the $z_2-$coordinate of the singularity. As a consequence, we obtain the desired estimate
\begin{align}
\int_{\mathbb{T}}\max_{1\leq j\leq N}\Biggl\{(1-|a_j(\zeta_2)|^2)^{1-\alpha}\prod_{j\neq k}\left|\frac{a_{j}(\zeta_2)-a_{k}(\zeta_2)}{1-\overline{a_{k}(\zeta_2)}a_{j}(\zeta_2)}\right|^2\Biggr\}|d\zeta_2|\gtrsim &\int_{U}\max_{1\leq j\leq N}(1-|a_j(\zeta_2)|^2)^{1-\alpha}|d\zeta_2|&&\\
\gtrsim&\int_{U}|\zeta_2-\tau_2|^{K(1-\alpha)}|d\zeta_2|.\notag
\end{align}
The last integral is known to converge if and only if $\alpha<1+\frac{1}{K}.$ This finishes the proof.
\end{proof}

In what follows, we will give a proof of the two corollaries stated in the main results section, namely Corollaries \ref{anisotropicchar} and \ref{sufficient0}. For this, we shall need the following  
\begin{proposition} Let $f\in \mathcal{O}(\mathbb{D}^2).$ Then, $f\in \mathcal{D}_{1,1,\alpha}(\mathbb{D}^2),$ if and only if $f\in \mathfrak{D}_{(\alpha,0)}(\mathbb{D}^2)\cap\mathfrak{D}_{(0,\alpha)}(\mathbb{D}^2).$
\label{claDir}
\end{proposition}
\begin{proof}
One can observe that
$$\int_{\mathbb{D}}\sup_{r>0}\int_{\mathbb{T}}|\partial_{z_1}f(z_1,re^{i\theta})|^2d\theta(1-|z_1|^2)^{1-\alpha} dA(z_1)=\int_{\mathbb{D}}\|\partial_{z_1}f(z_1,re^{i\theta})\|^2_{H^2(\mathbb{D})}(1-|z_1|^2)^{1-\alpha}dA(z_1)$$
Due to the Littlewood-Paley formula (see e.g. \cite{Yama}) and Fubini's Theorem, we obtain
\begin{align}
\int_{\mathbb{D}}\|\partial_{z_1}f(z_1,re^{i\theta})\|^2_{H^2(\mathbb{D})}dA_{\alpha}(z_1)=&\int_{\mathbb{D}}\left(\int_{\mathbb{D}}|\partial_{z_2}\partial_{z_1}f(z_1,z_2)|^2(1-|z_1|^2)dA(z_2)\right)dA_{\alpha}(z_1) \notag &&\\
=&\int_{\mathbb{D}^2}\left|\frac{\partial^2 f}{\partial z_1 \partial z_2}(z_1,z_2)\right|^2dA_{0}(z_1)dA_{\alpha}(z_2)
\end{align}
Repeating the same argument for the second integral, we arrive at
\begin{align}
\int_{\mathbb{D}}\|\partial_{z_2}f(re^{i\theta},z_2)\|^2_{H^2(\mathbb{D})}dA_{\alpha}(z_2)=&\int_{\mathbb{D}}\left(\int_{\mathbb{D}}|\partial_{z_1}\partial_{z_2}f(z_1,z_2)|^2(1-|z_1|^2)dA(z_1)\right)dA_{\alpha}(z_2) \notag &&\\
=&\int_{\mathbb{D}^2}\left|\frac{\partial^2 f}{\partial z_1 \partial z_2}(z_1,z_2)\right|^2dA_{0}(z_2)dA_{\alpha}(z_1)
\end{align}
Summing up we obtain the desired equivalence. 
\end{proof}
\begin{remark} One can immediately observe the following implication. If
$f\in \mathcal{D}_{1,1,\alpha}(\mathbb{D}^2)$ then $f\in \mathfrak{D}_{\frac{\alpha}{2}}(\mathbb{D}^2).$
\end{remark}
To see this, we apply the geometric-arithmetic mean inequality to get
\begin{align}
||f||^2_{\mathcal{D}_{1,1,\alpha}(\mathbb{D}^2)}=&\int_{\mathbb{D}^2}\left|\frac{\partial^2 f}{\partial z_1 \partial z_2}(z_1,z_2)\right|^2dA_{0}(z_1)dA_{\alpha}(z_2)+\int_{\mathbb{D}^2}\left|\frac{\partial^2 f}{\partial z_1 \partial z_2}(z_1,z_2)\right|^2dA_{\alpha}(z_1)dA_{1}(z_2) &&\\
\gtrsim&\int_{\mathbb{D}^2}\left|\frac{\partial^2 f}{\partial z_1 \partial z_2}(z_1,z_2)\right|^2(1-|z_1|^2)^{\frac{2-\alpha}{2}}(1-|z_2|^2)^{\frac{2-\alpha}{2}}dA(z_1)dA(z_2) \notag .
\end{align}
These two observations now allow us to prove Corollaries \ref{anisotropicchar} and \ref{sufficient0}.

Let us here provide the proof of Theorem \ref{derivativedirichlet}.

\begin{proof}[Proof of Theorem \ref{derivativedirichlet}] From our assumption on the bidegree of $p$, the function $\varphi_{\zeta_2}(z_1)$ is a single Blaschke factor in the $z_1-$variable , that is $$\varphi(z_1,\zeta_2)=\lambda(\zeta_2)\frac{z_1-a(\zeta_2)}{1-\overline{a(\zeta_2)}z_1}, \quad z_1\in\mathbb{D}, \zeta_2\in\mathbb{T}$$ for some $\lambda\in\mathbb{T}.$ 
This, in turn, implies that $$\frac{\partial^n{\varphi}}{\partial z_1^n}(z_1,\zeta_2)=\lambda(\zeta_2)\frac{\overline{a(\zeta_2)}^{n-1}(1-|a(\zeta_2)|^2)}{(1-\overline{a(\zeta_2)}z_1)^{n+1}}.$$
Now we will calculate the $\mathcal{D}_{n,m}-$norm of $\varphi.$ By the Rudin-Forelli estimates, 
\begin{align}
\int_{\mathbb{T}}\int_{\mathbb{D}}\left|\frac{\partial^n{\varphi}}{\partial z_1^n}(z_1,e^{i\theta})\right|^2dA(z_1)d\theta&=\int_{\mathbb{T}}\int_{\mathbb{D}}\left|\frac{\overline{a(e^{i\theta})}^{n-1}(1-|a(e^{i\theta})|^2)}{(1-\overline{a(e^{i\theta})}z_1)^{n+1}}\right|^2dA(z_1)d\theta &&\\
&= \int_{\mathbb{T}}|\overline{a(e^{i\theta})}|^{2n-2}(1-|a(e^{i\theta})|^2)^{2}\left(\int_{\mathbb{D}}\frac{dA(z_1)}{|1-\overline{a(e^{i\theta})}z_1|^{2n+2}}\right)d\theta\notag&&\\
&\approx \int_{\mathbb{T}}|\overline{a(e^{i\theta})}|^{2n-2}(1-|a(e^{i\theta})|^2)^{2-2n}d\theta \notag&&\\
&\approx \int_{\mathbb{T}}(1-|a(e^{i\theta})|^2)^{2-2n}d\theta \notag&&\\
&\approx \int_{U}|\tau_2-e^{i\theta}|^{K(2-2n)}d\theta, \notag
\end{align}
with the last integral converging if and only if $n<1+\frac{1}{2K}$. This implies that we have divergence of the integral whenever we take derivatives of order $n\geq 2$. On the other hand, assuming that $p$ is either of bidegree $(1,k)$ or $(k,1)$ for a positive integer $k>1,$ and $\varphi\in\mathcal{D}_{n,m},$ then the above estimates reveal that $n<2$ and $m<2$ respectively to each case, making the condition necessary. This observation finishes the proof.
\end{proof} 
\section{An alternative approach using Agler Kernels} \label{section 5}
In this section, we explore an alternative approach to the RIF membership problem by using Agler decompositions. In brief, membership of a RIF in a space of Dirichlet type is translated into integrability requirements for certain Agler kernels.

In 1988, Agler proved that every holomorphic function $f\colon\mathbb{D}^2\to\mathbb{D}$ posseses a decomposition of the form
$$1-f(z)\overline{f(w)}=(1-z_1\overline{w_1})K_1(z,w)+(1-z_2\overline{w_2})K_2(z,w),$$
where $K_1,K_2\colon\mathbb{D}^2\times\mathbb{D}^2\to \mathbb{C}$ are positive semi-definite kernel functions. For more on these kernels, see  \cite{AMY} and the references therein. Later, Knese \cite{KneseIndiana} gave a \textit{refined Agler decomposition} for holomorphic functions $f\colon\mathbb{D}^2\to\mathbb{D}.$ 
Let us provide more details here. Every holomorphic function $f\colon\mathbb{D}^2\to\mathbb{D}$ has a decomposition of the form
$$f(z)-f(w)=(z_1-w_1)L_{1}(z,w)+(z_2-w_2)L_2(z,w)$$ for all $z,w\in\mathbb{D}^2,$ 
where $L_1,L_2$ are again positive semi-definite kernel functions which satisfy \[|L_j(z,w)|^2\leq K_j(z,z)K_j(w,w)\quad  \textrm{and}\quad  L_j(z,z)=\frac{\partial f}{\partial z_j}(z).\] 

The proof of \cite[Theorem 10.2]{PLMS} combined the local Dirichlet integral, a two-variable Douglas formula, and Agler kernels. In the present article, we imitate their approach, utilizing the refined Agler Kernels from the paper \cite{KneseIndiana}.

Let us recall here some facts about the classical Dirichlet space of the unit disk. For a function $f\in H^2(\mathbb{D}),$ the Dirichlet integral $\int_{\mathbb{D}}|f'(z)|^2dA$ is equal to
$$\mathrm{Doug}(f):=\frac{1}{4\pi^2}\int_{\mathbb{T}^2}\left|\frac{f(\zeta)-f(\eta)}{\zeta-\eta}\right|^2|d\zeta||d\eta|.$$
See \cite[Chapter 1]{Ransford} for a proof of this fact.

In the case of the weighted Dirichlet-type spaces, a similar formula holds for the weighted Dirichlet integral, with the difference here being that we have an equivalence of semi-norms and not a strict equality. To be precise
$$\mathcal{D}_{a}(f)=\int_{\mathbb{D}}|f'(z)|^2(1-|z|^2)^{1-\alpha}dA(z)\approx\int_{\mathbb{T}^2}\frac{|f(\zeta)-f(\eta)|^2}{|\zeta-\eta|^{1+\alpha}}|d\zeta||d\eta|.$$
A similar formula is established in \cite{PLMS} for the space $\mathfrak{D}(\mathbb{D}^2)$. Here we also state and prove a weighted Douglas formula for the bidisc.
\begin{proposition} Let $0<\alpha\leq 1$ and suppose $f\in\mathfrak{D}_{\alpha}(\mathbb{D}^2)$. Then 
\begin{align}
\mathrm{Doug}_{a}(f)=|f(0,0)|^2+&\int_{\mathbb{T}^2}\frac{|f(\zeta_1,0)-f(\eta_1,0)|^2}{|\zeta_1-\eta_1|^{1+\alpha}}|d\zeta_1||d\eta_1| &&\\
                    +&\int_{\mathbb{T}^2}\frac{|f(0,\zeta_1)-f(0,\eta_1)|^2}{|\zeta_2-\eta_2|^{1+\alpha}}|d\zeta_2||d\eta_2| \notag &&\\
                    +&\int_{\mathbb{T}^4}\frac{|f(\zeta_1,\zeta_2)-f(\zeta_1,\eta_2)+f(\eta_1,\eta_2)-f(\eta_1,\zeta_2)|^2}{|\zeta_1-\eta_1|^{1+\alpha}|\zeta_2-\eta_2|^{1+\alpha}}|d\zeta||d\eta| \notag
\end{align}
is an equivalent Dirichlet norm for the space $\mathfrak{D}_{\alpha}(\mathbb{D}^2)$
\label{weighteddouglas}
\end{proposition}
\begin{proof} Obviously $|f(0,0)|=|a_{00}|.$ Fixing $z_2=0$ for the second integral, and along with the fact that $f(\cdot,0)\in \mathfrak{D}_{\alpha}(\mathbb{D}),$ we obtain 
$$\int_{\mathbb{T}^2}\frac{|f(\zeta_1,0)-f(\eta_1,0)|^2}{|\zeta_1-\eta_1|^{1+\alpha}}|d\zeta_1||d\eta_1|\approx\sum_{k=1}^{\infty}k^{\alpha}|a_{k0}^2|.$$
We argue similarly for $f(0,\cdot).$ For the last term, we observe that after changing the variables to $\zeta_j=e^{i(s_j+t_j)}, \eta_j=e^{it_j}$ for $j=1,2,$ we obtain
\begin{align}
&\int_{\mathbb{T}^4}\frac{|f(\zeta_1,\zeta_2)-f(\zeta_1,\eta_2)+f(\eta_1,\eta_2)-f(\eta_1,\zeta_2)|^2}{|\zeta_1-\eta_1|^{1+\alpha}|\zeta_2-\eta_2|^{1+\alpha}}|d\zeta||d\eta|= &&
\\
&\int_{[0,2\pi]^2}\int_{[0,2\pi]^2}\frac{|f(e^{i(s_1+t_1)},e^{i(s_2+t_2)})-f(e^{i(s_1+t_1)},e^{it_2})-f(e^{it_1},e^{i(s_2+t_2)})+f(e^{it_1},e^{it_2})|^2}{|e^{is_1}-1|^{1+\alpha}|e^{is_2}-1|^{1+\alpha}}dtds. \notag 
\end{align}
Applying Parseval's formula twice, we obtain
$$\frac{1}{4\pi^2}\sum_{k,\ell=0}^{\infty}|a_{k\ell}|^2\int_{[0,2\pi]^2}\frac{|e^{is_1k}-1|^2|e^{is_2\ell}-1|^2}{|e^{is_1}-1|^{1+\alpha}|e^{is_2}-1|^{1+\alpha}}ds\approx\sum_{k,\ell=0}^{\infty}k^{\alpha}\ell^{\alpha}|a_{k\ell}|^2.$$
and the proof is complete.
\end{proof} 

For what follows, we will refer to the kernel functions $L_j$ as refined Agler kernels. To lighten notation, we introduce the notation $L_{j}((z_1,z_2),(z_1,w_2))=L_j(z_1,z_2,w_2)$.

We lay out our approach in the following series of propositions.
\begin{proposition} Let $\varphi=\frac{\tilde{p}}{p}$ a RIF and $L_j$ its refined Agler Kernels. Then $\varphi\in \mathcal{D}_{1,1,\alpha}$ if and only if
$|\zeta_j-\eta_j|^{\frac{1-\alpha}{2}}L_{j}(\zeta_1,\zeta_2,\eta_j)\in L^2(\mathbb{T}^3), j=1,2.$
\label{5.1.}
\end{proposition}
\begin{proof} The main tool for the proof of this result is the one dimensional weighted Douglas Formula. To see this, observe that 
$$\int_{\mathbb{T}}\int_{\mathbb{D}}|\partial_{z_1}\varphi(z_1,e^{i\theta})|^2(1-|z_1|^2)^{1-\alpha}dA(z_1)d\theta=\int_{\mathbb{T}}\|\varphi_{\zeta_2}(z_1)\|^2_{\mathcal{D}_{\alpha}(\mathbb{D})}d\theta.$$
It is well known that
$$\|\varphi_{\zeta_2}(z_1)\|^2_{\mathcal{D}_{\alpha}(\mathbb{D})}\approx\int_{\mathbb{T}^2}\frac{|\varphi_{\zeta_2}(\zeta_1)-\varphi_{\zeta_2}(\eta_1)|^2}{|\zeta_1-\eta_1|^{1+\alpha}}|d\zeta_1||d\eta_1|.$$
Now, we observe that by the refined Agler decomposition
$$\varphi_{\zeta_2}(\zeta_1)-\varphi_{\zeta_2}(\eta_1)=(\zeta_1-\eta_1)L_1(\zeta_1,\zeta_2,\eta_1),$$
$\zeta_1,\zeta_2,\eta_1 \in\mathbb{T}.$ Substituting into the one-dimensional Douglas Formula, we obtain
\begin{align}
\int_{\mathbb{T}}\int_{\mathbb{D}}|\partial_{z_1}\varphi(z_1,e^{i\theta})|^2(1-|z_1|^2)^{1-\alpha}dA(z_1)d\theta\approx &\int_{\mathbb{T}}\left(\int_{\mathbb{T}^2}\frac{|\varphi_{\zeta_2}(\zeta_1)-\varphi_{\zeta_2}(\eta_1)|^2}{|\zeta_1-\eta_1|^{1+\alpha}}|d\zeta_1||d\eta_1|\right)|d\zeta_2|&&\\
=&\int_{\mathbb{T}}\left(\int_{\mathbb{T}^2}\frac{|\zeta_1-\eta_1|^2|L_1(\zeta_1,\zeta_2,\eta_1)|^2}{|\zeta_1-\eta_1|^{1+\alpha}}\right)|d\zeta_2| \notag&&\\
=&\int_{\mathbb{T}^3}|\zeta_1-\eta_1|^{1-\alpha}|L_1(\zeta_1,\zeta_2,\eta_1)|^2|d\zeta_1||d\zeta_2||d\eta_1|\notag 
\end{align}
and the result follows.
\end{proof}
\begin{corollary} Let $\varphi=\frac{\widetilde{p}}{p}$ a RIF and let $L_j$ be its Refined Agler Kernels. Then $L_{j}(\zeta_1,\zeta_2,\eta_1)\in L^2(\mathbb{T}^3), j=1,2.$
\end{corollary}
\begin{proof} From Corollary \ref{easy}, we know that every RIF belongs to the Dirichlet-type space $\mathcal{D}_{1,1,1}(\mathbb{D}^2).$ This, in turn, implies that the functions $L_{j}(\zeta_1,\zeta_2,\eta_1)\in L^2(\mathbb{T}^2).$ 
\end{proof}
Next we prove a similar proposition for the classical weighted Dirichlet spaces $\mathfrak{D}_{\alpha}(\mathbb{D}^2)$. This time we apply the weighted Douglas Formula on the bidisc.
\begin{proposition}
Let $\varphi$ be a RIF and $L_j$ its Refined Agler Kernels. Then $\varphi\in\mathfrak{D}_{a}(\mathbb{D}^2)$ if and only if
$$\int_{\mathbb{T}^4}\frac{|\zeta_2-\eta_2|^{1-\alpha}}{|\zeta_1-\eta_1|^{1+\alpha}}|L_2(\zeta_1,\zeta_2,\eta_2)-L_2(\eta_1,\eta_2,\zeta_2)|^2|d\zeta||d\eta|<+\infty$$
\label{prop}
\end{proposition}
\begin{proof} The main tool of the proof is the refined Agler decomposition and the weighted Douglas Formula for the weighted Dirichlet space of the bidisc we proved above. Set $(z_1,z_2)=(\zeta_1,\zeta_2)\in\mathbb{T}^2$ and $(w_1,w_2)=(\eta_1,\eta_2)\in\mathbb{T}^2$. From the refined Agler decomposition we obtain
$$\varphi(\zeta_1,\zeta_2)-\varphi(\zeta_1,\eta_2)=(\zeta_2-\eta_2)L_2(\zeta_1,\zeta_2,\eta_1),$$
and
$$\varphi(\eta_1,\eta_2)-\varphi(\eta_1,\zeta_2)=(\eta_2-\zeta_2)L_2(\eta_1,\eta_2,\zeta_2),$$
for all $\zeta,\eta\in \mathbb{T}^2.$
Summing up,
\begin{align}
\varphi(\zeta_1,\zeta_2)-\varphi(\zeta_1,\eta_2)+\varphi(\eta_1,\eta_2)-\varphi(\eta_1,\zeta_2)=(\zeta_2-\eta_2)(L_2(\zeta_1,\zeta_2,\eta_1)-L_2(\eta_1,\eta_2,\zeta_2))
\end{align}
for $\zeta_1,\zeta_2,\eta_1,\eta_2\in \mathbb{T}.$
The result follows after substituting into the weighted Douglas norm of $\mathfrak{D}_{a}(\mathbb{D}^2).$
\end{proof}
\begin{remark}
    In the weighted case, the local Dirichlet integral formula does not hold even in the one dimensional setting. This poses the most significant obstacle in this method.
\end{remark}
\section{Examples} \label{section 6}
In this section we showcase that the standard example 
\[\kappa(z_1,z_2)=\frac{2z_1z_2-z_1-z_2}{2-z_1-z_2}, \quad z\in \mathbb{D}^2,\]
is in $\mathcal{D}(\mathbb{D}^2),$ but fails to belong to  $\mathcal{D}_{n,m}(\mathbb{D}^2)$ whenever $\max(n,m)\ge 2.$ Moreover, we prove that the AMY function belongs to $\mathcal{D}_{1,1,\alpha}(\mathbb{D}^2)$ if and only if $\alpha<\frac{5}{4}.$
These examples are covered by our main results but we here illustrate how to recover membership by direct computations.
\begin{example} $\kappa \in \mathcal{D}(\mathbb{D}^2).$ 
\end{example}
\begin{proof} Due to symmetry of the RIF, it suffices to prove that
$$\lim_{r\to1^{-}}I_r(\kappa)=\lim_{r\to 1^{-}}\int_{\mathbb{T}}\left(\int_{\mathbb{D}}\left|\frac{\partial \kappa}{\partial z_1}(z_1,re^{i\theta})\right|^2dA(z_1)\right)d\theta<+\infty.$$
Calculations show that 
$$\frac{\partial \kappa}{\partial z_1}(z_1,z_2)=\frac{-2(z_2-1)^2}{(2-z_1-z_2)^2}$$
for all $(z_1,z_2)\in\mathbb{D}^2.$
Hence 
$$I_r(\kappa)=\int_{\mathbb{T}}|re^{i\theta}-1|^4\left(\int_{\mathbb{D}}\frac{1}{|2-re^{i\theta}-z_1|^4}dA(z_1)\right)d\theta$$
Set $w=2-re^{i\theta}$ and observe that $|w|\ge 1$ which implies that $\frac{1}{|w|}\leq 1.$ Factoring out $w$ from the denominator in the inner integral we get
$$I_r(\kappa)=\int_{\mathbb{T}}\frac{|re^{i\theta}-1|^4}{|w|^4}\left(\int_{\mathbb{D}}\frac{1}{\left|1-\frac{z_1}{w}\right|^4}dA(z_1)\right)d\theta.$$
We now observe that
$$\int_{\mathbb{D}}\frac{1}{\left|1-\frac{z_1}{w}\right|^4}dA(z_1)=\langle k_{\frac{1}{\overline{w}}}(z_1),k_{\frac{1}{\overline{w}}}(z_1)\rangle=\frac{1}{\left(1-\left|\frac{1}{w}\right|^2\right)^2}=\frac{|w|^4}{(|w|^2-1)^2}.$$
Substituting inside the integral yields
$$I_r(\kappa)=\int_{0}^{2\pi}\frac{|re^{i\theta}-1|^4}{(|2-re^{i\theta}|^2-1)^2}=\int_{0}^{2\pi}\frac{(r^2-2r\cos \theta+1)^2}{(r^2-4r\cos \theta+3)^2}d\theta<+\infty$$
and the result follows. 
\end{proof}
\begin{example} $\kappa \notin \mathcal{D}_{n,m}(\mathbb{D}^2),$ for $\max(n,m)\ge 2.$
\end{example}
\begin{proof}
We shall follow the approach and the calculations in \cite[Example 12.1]{PLMS}. We will briefly navigate the reader through the calculations by writing out the important estimates. We will ommit the majority of the details, as the interested reader can find them on \cite{PLMS}. The function $\kappa$ has Taylor series representation
$$\kappa(z_1,z_2)=a_{00}+\sum_{k,\ell=0}^{\infty}a_{(k+1)(\ell+1)}z_1^kz_2^{\ell},$$
and its coefficients satisfy $$a_{(k+1)(\ell+1)}=\binom{k+\ell}{k}2^{-(k+\ell)}-\binom{k+\ell+2}{k+1}2^{-(k+\ell+2)}\approx \frac{1}{(k+\ell)^\frac{3}{2}},$$
with the last equality being an application of Stirling's approximation. For sufficient large indices $k,\ell,$ we take the tail of the sum and we find that
$$\|\kappa\|^2_{\mathcal{D}_{n,m}(\mathbb{D}^2)}\ge \sum_{k\ge N}^{\infty}\left(k^{2n-1}+k^{2m-1}\right)\frac{1}{k^{5/2}}\approx\sum^{\infty}_{k\ge N}\frac{1}{k^{-2n+1+5/2}}+\sum_{k\ge N}^{\infty}\frac{1}{k^{-2m+1+5/2}},$$
for $N$ sufficiently large.
Immediately, one observes that we lose convergence whenever at least one of $n,m$ are greater than $\frac{5}{4}.$ Restricting to integer values, we see that we lose membership whenever we differentiate $\kappa$ a second time in either $z_1$ or $z_2.$ 
\end{proof}
\begin{example}
Consider $\kappa=\frac{2z_1z_2-z_1-z_2}{2-z_1-z_2}.$ If $\alpha<1/2,$ then $\varphi\in\mathfrak{D}_{\alpha}(\mathbb{D}^2)$
\end{example}
In this example we will consider the refined Agler decomposition of $\kappa,$ following the approach of Proposition \ref{5.1.}. We will prove that for all $2\alpha<1,$ the RIF $\kappa$ belongs to $\mathcal{D}_{1,1,2\alpha}(\mathbb{D}^2).$  We have
$$L_{1}(\zeta_1,\zeta_2,\eta_1)=\frac{2(\zeta_2-1)(1-\zeta_2)}{(2-\zeta_1-\zeta_2)(2-\eta_1-\zeta_2)}=\frac{-2(1-\zeta_2)^2}{(2-\zeta_1-\zeta_2)(2-\eta_1-\zeta_2)}.$$
Now consider $\alpha<1/2.$ Then $1-2\alpha>0$ and the term $|\zeta_1-\eta_1|^{1-2\alpha}\leq 2^{1-2\alpha}$ by the triangle inequality. Hence,
\begin{align}\int_{\mathbb{T}^3}|\zeta_1-\eta_1|^{1-2a}|L_1(\zeta_1,\zeta_2,\eta_1)|^2|d\zeta_1||d\zeta_2||d\zeta_3|=&\int_{\mathbb{T}^3}\frac{4|\zeta_1-\eta_1|^{1-2\alpha}|1-\zeta_2|^4}{|2-\zeta_1-\zeta_2|^2|2-\eta_1-\zeta_2|^2}|d\zeta_1||d\zeta_2||d\zeta_3|&&\\
\lesssim &\int_{\mathbb{T}^3}\frac{|1-\zeta_2|^4|d\zeta_1||d\zeta_2||d\eta_1|}{|2-\zeta_1-\zeta_2|^2|2-\eta_1-\zeta_2|^2}\notag&&\\ 
\approx&\int_{\mathbb{T}}|1-\zeta_2|^4\left(\int_{\mathbb{T}}\frac{|d\zeta_1|}{|2-\zeta_1-\zeta_2|^2}\right)^2|d\zeta_2|\notag&&\\
\approx&\int_{\mathbb{T}}\frac{|1-\zeta_2|^4}{(|2-\zeta_2|^2-1)^2}|d\zeta_2|\notag
\end{align}
Note that $|1-\zeta_2|^4=|1-e^{i\theta}|^4=(2|\sin(\theta/2)|)^4=16\sin^{4}(\theta/2).$ Hence
$$\int_{\mathbb{T}}\frac{|1-\zeta_2|^4}{(|2-\zeta_2|^2-1)^2}|d\zeta_2|=\int_{0}^{2\pi}\frac{16\sin^4(\theta/2)}{64\sin^4(\theta/2)}d\theta=\frac{1}{4}<\infty.$$
This means that $\kappa\in\mathcal{D}_{1,1,2\alpha}(\mathbb{D}^2)$ for all $\alpha<1.$ By Proposition \ref{prop} then, we obtain that $\kappa\in \mathfrak{D}_{\alpha}(\mathbb{D})$ for all $\alpha<1/2.$
It is worth noting here that the sharp constant $\alpha$ for which we have membership of $\kappa$ in the spaces $\mathfrak{D}_{\alpha}$ is $\alpha=3/4$ (see  \cite{PLMS} example 12.1). Nevertheless this method aligns with the fact that all RIFs belong to $\mathcal{D}_{1,1}(\mathbb{D}^2)$ and, as a consequence, to $\mathfrak{D}_{1/2}(\mathbb{D}^2).$
\begin{example} (AMY example) Let $$\psi(z_1,z_2)=\frac{4z_1^2z_2-z_1^2-3z_1z_2-z_1+z_2}{4-3z_1-z_2-z_1z_2+z_1^2}.$$ Then $\psi\in \mathcal{D}_{1,1,\alpha}(\mathbb{D}^2)$ if and only if $\alpha<\frac{5}{4}$
\end{example}
We need to estimate, from above and below, the integrals 
$I_1=\int_{\mathbb{T}}\int_{\mathbb{D}}\left|\frac{\partial \psi}{\partial z_1}(z_1,\zeta_2)\right|^2dA_{\alpha}(z_1)d\theta$ and $I_2=\int_{\mathbb{T}}\int_{\mathbb{D}}\left|\frac{\partial \psi}{\partial z_1}(\zeta_1,z_2)\right|^2dA_{\alpha}(z_2)d\theta.$ The sufficiency part is just an application of the sufficiency part of Theorem \ref{membership}. For the necessity part, one needs to study the branches of the zero set of $\psi$ for both cases, i.e. whenever the $z_1-$variable is fixed on the circle and whenever the $z_2-$variable is fixed on the circle.

The branches of the zero set of the slice $\psi_{\zeta_2}(z_1)$ are
\[z_1=a_1(z_2)=\frac{3z_2+1 + \sqrt{-7z_2^2+10z_2+1}}{2(4z_2-1)}\quad 
\textrm{and} \quad 
z_1=a_2(z_2)=\frac{3z_2+1 - \sqrt{-7z_2^2+10z_2+1}}{2(4z_2-1)};\]
see Figure 1 for a visualization.
The contact order of the singularity $(1,1)$ for the AMY function is well known to be $K=4$, viz. \cite[Example 2]{PLMS}.

We estimate the integral $I_1$ from below and obtain
\begin{align}
\int_{\mathbb{T}}\int_{\mathbb{D}}\left|\frac{\partial \psi}{\partial z_1}(z_1,\zeta_2)\right|^2dA_{\alpha}(z_1)d\theta &\ge \max_{j=1,2}\int_{\mathbb{T}}(1-|a_j(\zeta_2)|^2)^{1-\alpha}\rho_{\mathbb{D}}(a_1(\zeta_2),a_2(\zeta_2))|d\zeta_2|&&\\
&\approx \int_{U}|1-\zeta_2|^{4(1-\alpha)}|d\zeta_2|\notag
\end{align}
\begin{figure}[h!]
\begin{center}
\includegraphics[width=14cm ]{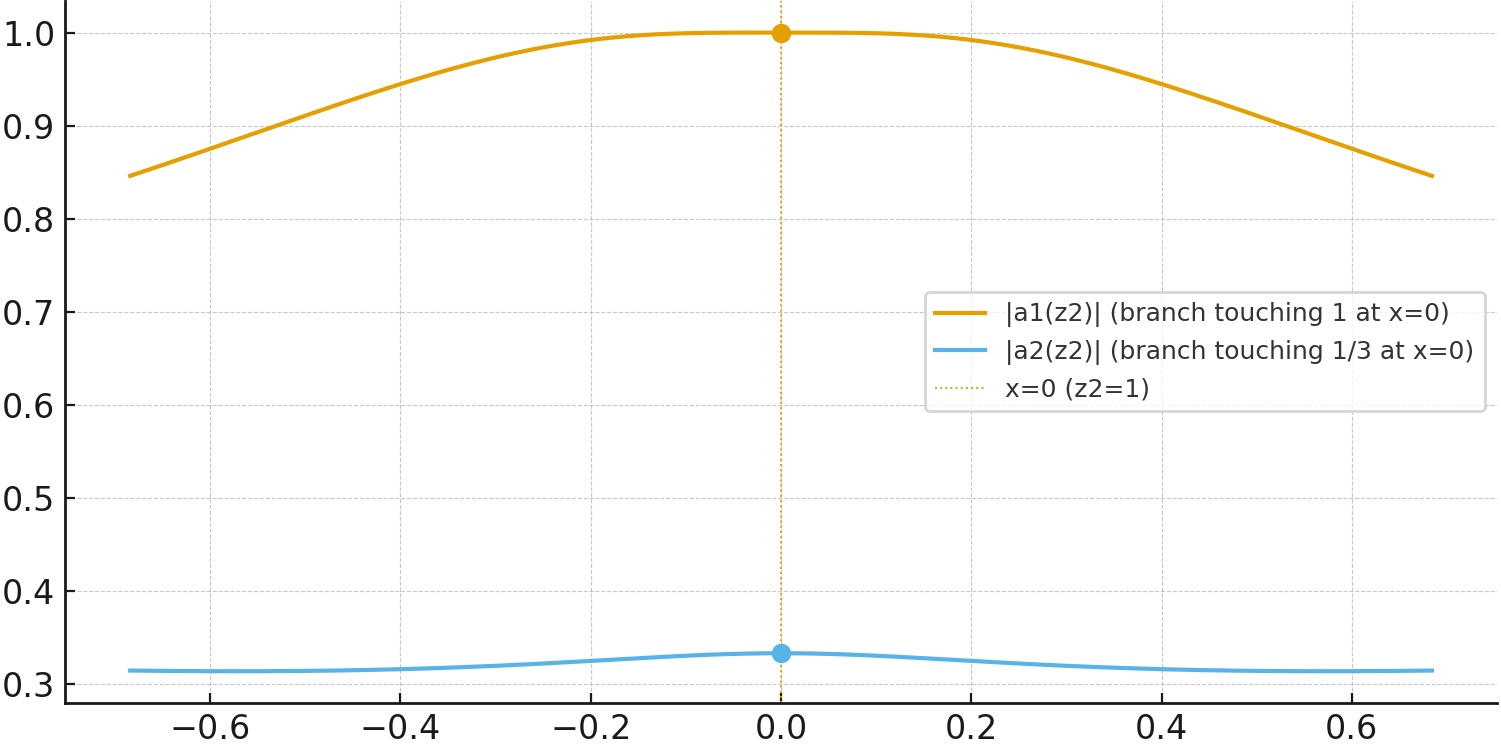}
\caption{Graphs of the functions $|a_1(\zeta_2)|$ and $|a_2(\zeta_2)|$ on a neighborhood of the origin after mapping the $\zeta_2-$variable on the real line. We observe that as $z_2\to 1,$ the hyperbolic pseudodistance $\rho_{\mathbb{D}}(a_1(\zeta_2),a_2(\zeta_2))\to 1$ because $a_1(1)=1$ and $a_2(1)=\frac{1}{3}$.}
\end{center}
\end{figure}
where $U$ is a neighborhood around $1.$ This integral evidently converges if and only if $\alpha<\frac{5}{4}.$

The integral $I_2$ is associated with a one-branch zero set, corresponding to having only one Blaschke factor. The reader is invited to check that the integral $I_2$ converges also if and only if $\alpha<\frac{5}{4}.$ Combining these with the sufficiency estimates, we get a sharp cut-off for $\alpha$ and membership of the AMY function in $\mathcal{D}_{1,1,\alpha}.$
\section{Acknowledgements} The first named author was financially supported by the National
Science Center, Poland, SHENG III, research project 2023/48/Q/ST1/00048. Moreover, the first named author acknowledges financial support for his research stay in Stockholm University from the program Excellence Initiative at the Jagiellonian University in Krak\'ow `Research Support Module 2025."  

We sincerely thank the anonymous referee for a very careful reading of the paper, and for making numerous and detailed suggestions on how to improve it. In particular, Lemma 4.3 was suggested to us by the referee.

\end{document}